\documentclass[12pt,reqno]{amsart}

\usepackage{amssymb}
\usepackage{amsmath}
\usepackage{latexsym}
\usepackage{amsthm}
\usepackage{epsfig}
\usepackage{amsfonts}
\usepackage{amssymb}
\usepackage{amsmath}
\usepackage{latexsym}
\usepackage{amsthm}
\usepackage{epsfig}
\usepackage{color}
\usepackage{amsfonts,amssymb,mathrsfs,amscd}
\usepackage{comment}
\usepackage{bm}
\newtheorem{proposition}{Proposition}[section]
\newtheorem{theorem}[proposition]{Theorem}
\newtheorem{corollary}[proposition]{Corollary}

\theoremstyle{definition}

\newtheorem{remark}[proposition]{Remark}

\numberwithin{equation}{section}
\newcommand{\rot}{\mathop\mathrm{rot}}

\renewcommand{\div}{\mathop\mathrm{div}}

\newcommand{\Tr}{\mathrm{Tr}}

\begin{document}
\title
[ Navier--Stokes--Voight system ]
 {Attractors for the Navier--Stokes--Voight equations and their dimension }

 \author[ A. Ilyin,
and  S. Zelik] { Alexei Ilyin${}^1$ and Sergey Zelik${}^{2,3,4}$}



\email{ilyin@keldysh.ru} \email{s.zelik@surrey.ac.uk}
\address{${}^1$ Keldysh Institute of Applied Mathematics, Moscow, Russia}
\address{${}^2$ Zhejiang Normal University, Department of Mathematics, Zhejiang, China}
\address{${}^3$ University of Surrey, Department of Mathematics, Guildford, GU2 7XH, United
Kingdom.}
\address{${}^4$ HSE University, Nizhny Novgorod, Russia}

\maketitle

\begin{abstract}
The Voight regularization of the Navier--Stokes sys\-tem is studied
in a bounded domain and on the torus. In the 3D case we obtain new
explicit bounds for the attractor dimension improving the
previously known results. In the 2D case we show that the estimates
so obtained converge to the known estimates for the attractor of
the Navier--Stokes system  as the regularization parameter tends to
zero both for the Dirichlet and the periodic boundary conditions.
\end{abstract}
\medskip

\textit{To A.I.\,Aptekarev on the occasion of his 70th birthday}

 \setcounter{equation}{0}

\section{Introduction} \label{sec1}

One of the crucial characteristics of a turbulent fluid flow
is the large range of spatial and temporal scales.
This characteristic property is a source of difficulties both in theoretical studies and in practical calculations. Moreover, in many  applications in practice, physically significant
characteristics are often concentrated on large spatial scales, as seen, for example, in numerical
hydrodynamic weather forecasting. Therefore, much effort has been put into modeling  large-scale
dynamics of turbulent flow by filtering out smaller scales.

Typically, such filtering is done by applying the operator $(1-\alpha\Delta)^{-1}$ to the first or second argument of the bilinear Navier--Stokes operator (or to the whole operator).
The parameter $\alpha$ has the dimension $\text{length}^2$ and determines the scale at which
high-frequency spatial modes will be filtered out.
The corresponding regularized systems are usually called alpha
models.

These models have recently  attracted much attention  both from the point of view of
 numerical modeling and
 the theory of attractors.
Without claiming to be complete, we point out the works
\cite{ CaoLunTiti,ChepMS,ChepTitiVishik,CHOT,CotiGal,FHT, TitiVarga}
and the references therein. Typical
problems from the point of view of attractors  are the construction of the attractor of the
regularized system,
estimates of its dimension and the study of  the singular limits $\alpha\to0$,
in other words, proving   its convergence
to a weak attractor of the initial three-dimensional system.

One of these models is the Navier--Stokes--Voight model:
\begin{equation}\label{DEalpha}
\left\{
  \begin{array}{ll}
    (1-\alpha \Delta)\partial_t u+( u,\nabla u) u-\nu\Delta u+\nabla p=g,\  \  \\
    \operatorname{div}  u=0,\quad u(0)=u_0,\\
    u\vert_{\partial\Omega}=0,
  \end{array}
\right.
\end{equation}
where $g$ is the right-hand side, and $\nu$ is the kinematic viscosity coefficient.

The system is studied
\begin{enumerate}
  \item on the torus $\Omega=\mathbb{T}^d=[0,2\pi]^d$. In this case, the standard   zero mean
      condition is assumed for $u$ and $g$;
  \item in a bounded domain  $\Omega\subset\mathbb{R}^d$ with smooth boundary and with
  Dirichlet boundary condition for    $u$,
\end{enumerate}
where $d=2,3$.

The well-posedness
of this system was established in \cite{Oskol}.
Estimates of the number of degrees of freedom, expressed in terms of the dimension of the attractor and the number of determining  functionals, were first obtained in \cite{TitiVarga}, and then
these estimates were improved and refined in \cite{CotiGal}.

In our work, in the next section, the system \eqref{DEalpha}
is reduced to an equation with bounded operator coefficients
in the Sobolev space $\mathbf H^1$, and the existence of an
attractor $\mathscr A=\mathscr A_\alpha$ is proved.
Note that, unlike the the classical Navier--Stokes system,
 we now have one more independent dimensionless parameter (in addition to the Grashof num\-ber $G$).
Namely,  the product $\alpha\lambda_1$ (where, as usual, $\lambda_1$ is the first
eigenvalue of the Stokes operator). For this reason, all estimates
of the dimension are naturally expressed as functions of these two
basic quantities. This additional parameter
can equivalently be written in the form  $\alpha/|\Omega|$
which is more convenient to use in some estimates in the two-dimensional case.

In this work, two estimates of the fractal dimension of the attractor are derived with different asymptotic behaviour with respect to  two large dimensionless parameters $(\alpha\lambda_1)^{-1}$ and $G$ , which can be combined in the following symmetric form
$$
\dim_F\mathscr A_\alpha\preceq(\alpha\lambda_1)^{-3/4}\cdot G^2
\min\bigl[(\alpha\lambda_1)^{-3/4}, G^2\bigr],\quad d=3,\
 G=\frac{\|g\|}{\lambda_1^{3/4}\nu^2}.
$$

In section~\ref{sec3} we consider the two-dimensional case and
for both types
of boundary conditions obtain the estimate
$$
\dim_F\mathscr A_\alpha\le\frac{(\alpha\lambda_1+1)\mathrm{c}_{\mathrm {LT}}}2G^2,
\quad d=2,\
G=\frac{\|g\|}{\lambda_1\nu^2},
$$
where $\mathrm c_{\mathrm{LT}}$ is the corresponding Lieb--Thirring constant

Moreover, if $\alpha\in [0,|\Omega|G^{-1}]$,
then the estimates go over to the well-known estimates for the
classical Navier--Stokes system:
$$
\dim_F\mathscr A_\alpha\preceq G,\quad
\dim_F\mathscr A_\alpha\preceq G^{2/3}(\ln(1+G))^{1/3}
$$
for the Dirichlet and periodic boundary conditions, respectively.
All the constants here are independent of  $\alpha$ and are given
in an explicit form including the case $\alpha=0$.

Finally, in section~\ref{sec4}  we prove a variant of a periodic
Brezis--Gallouet inequality in the suborthonormal case.

 \setcounter{equation}{0}
\section{Attractor of the Navier--Stokes--Voight system } \label{sec2}

We write the system as an evolution equation in the Hilbert space
$\mathbf H^1=\{H^1(\Omega)\}^d$ with  $\div u=0$ and the Dirichlet
boundary condition
or the zero mean condition if $\Omega\Subset \mathbb R^d$ or
$\Omega=\mathbb T^d$, respectively. For this purpose, we consider the Stokes system
$$
\left\{
  \begin{array}{ll}
  u-\alpha\Delta u+\nabla q=f,\  \  \\
  \div u=0,\\
    u\vert_{\partial\Omega}=0,
  \end{array}
  \right.
$$
and denote by $u=(1+\alpha A)^{-1}f$ its solution.
(If $\Omega=\mathbb {T}^d$, then there are no boundary conditions and
 $\nabla q\equiv0$.)
Here  $A:=-\Pi\Delta$ is the Stokes operator and  $\Pi$ is
the Helmholtz--Leray projection onto the subspace
of divergence free vector fields in the corresponding domain.

It is convenient to define the scalar product in $\mathbf H^1$ as follows
$$
(u,v)_\alpha:=(u,v)+\alpha(\nabla u,\nabla v)=(u,(1+\alpha A)v).
$$

We apply to~\eqref{DEalpha} the operator $(1+\alpha A)^{-1}$. We obtain
\begin{equation}\label{NSV}
\partial u+\nu A(1+\alpha A)^{-1}u+(1+\alpha A)^{-1}B(u,u)=(1+\alpha A)^{-1}g,
\end{equation}
where
$$
B(u,v)=\Pi\bigl((u,\nabla)v\bigr)
$$
is the standard Navier--Stokes bilinear operator  \cite{BV,T, ZUMN}.

Equation~\eqref{NSV} is an equation with bounded operator coefficients~\cite{IKZ22}, hence there exists a unique local in time solution, which, in fact, exists globally as the following
a priori estimate shows.
\begin{proposition}\label{Prop:1}
The solution $u(t)$ satisfies the estimate
\begin{equation}\label{apr1}
\|u(t)\|_\alpha^2\le\|u(0)\|_\alpha^2e^{-\gamma t}
+\frac{(\alpha\lambda_1+1)}{\nu^2\lambda_1^2}\|g\|^2\left(1-e^{-\gamma t}\right),
\end{equation}
where $\lambda_1>0$ is the first eigenvalue of the Stokes operator in $\Omega$ or on the torus
with zero mean condition, and $\gamma=\nu\lambda_1/(\alpha\lambda_1+1)$.
\end{proposition}

\begin{proof}
We take the scalar product of  \eqref{NSV} and $(1+\alpha A)u$. We obtain
$$
\aligned
&\partial_t\|u(t)\|^2_\alpha+2\nu\|\nabla u(t)\|^2\le2(g,u)\\\le
&\nu{\lambda_1}\|u\|^2+\frac1{\lambda_1\nu}\|g\|^2\le
\nu\|\nabla u\|^2+\frac1{\lambda_1\nu}\|g\|^2,
\endaligned
$$
or
\begin{equation}\label{scalpr}
\partial_t\|u(t)\|^2_\alpha+\nu\|\nabla u(t)\|^2\le
\frac1{\lambda_1\nu}\|g\|^2.
\end{equation}

Next, using the Poincar\'e inequality we see that
$$
\|\nabla u\|^2\ge\frac1{\alpha+\lambda_1^{-1}}\|u\|_\alpha^2,
$$
which together with \eqref{scalpr} gives~\eqref{apr1}.
\end{proof}

Thus,  equation \eqref{NSV} has a unique solution defining thereby in $\mathbf H^1$
the semigroup $S(t)u_0=u(t)$, which is dissipative. Following \cite{IZLap70} or
\cite{IKZ22}, it is easy to see that the semigroup $S(t)$ has a global attractor
$\mathscr A_\alpha\Subset\mathbf H^1$. This follows in a standard way by splitting the semigroup
$S(t)$ into an exponentially contracting and uniformly compact parts
taking into account that the nonlinear operator $(1-\alpha A)^{-1}B(u,u)$
is bounded from $\mathbf H^1$ to $\mathbf H^{3/2}$ (for $d=3$).

We will need the following estimate for solutions
lying on the attractor.

\begin{proposition}
It holds for  $u(t)\in \mathscr A_\alpha$ that
\begin{equation}\label{apr2}
\limsup_{t\to\infty}\frac1t\int_0^t
\|\nabla u(\tau)\|d\tau\le\frac1{\sqrt{\lambda_1}\nu}\|g\|.
\end{equation}
\end{proposition}

\begin{proof}
Integrating \eqref{scalpr} in $t$ and using that
$\|u(t)\|_\alpha^2$ is bounded in view of~\eqref{apr1},
we obtain
\begin{equation}\label{u}
\limsup_{t\to\infty}\frac1t\int_0^t\|\nabla u(\tau)\|^2d\tau\le
\frac1{\lambda_1\nu^2}\|g\|^2,
\end{equation}
which by H\"older's inequality
$$
\frac1t\int_0^t\|\nabla u(\tau)\|\,d\tau\le
 \left(\frac1t\int_0^t\|\nabla  u(\tau)\|^2\,d\tau\right)^{1/2}
$$
gives~\eqref{apr2}.
\end{proof}

 We now turn to our main goal, namely, to estimating the fractal dimension of the attractor $\mathscr A_\alpha$.

\begin{theorem}\label{Th:1}
The attractor $\mathscr A_\alpha$ has a finite fractal dimension in $\mathbf H^1$:
\begin{equation}\label{dim23}
\aligned
&\dim_F\mathscr A_\alpha\le\frac{(\alpha\lambda_1+1)^2}{8\pi\alpha\lambda_1}G^2,
\quad G=\frac{\|g\|}{\lambda_1\nu^2},\ d=2,\\
&\dim_F\mathscr A_\alpha\le\frac{(\alpha\lambda_1+1)^2}{6\pi(\alpha\lambda_1)^{3/2}}G^2
,\quad G=\frac{\|g\|}{\lambda_1^{3/4}\nu^2},\ d=3.
\endaligned
\end{equation}
\end{theorem}
\begin{proof}
The semigroup $S(t): {\bf H}^1\to
{\bf H}^1$ depends smoothly on the initial data, and we
only need to estimate the $n$-trace of the  equation \eqref{NSV} linearized on the solution
$u(t)\in \mathscr A_\alpha$:
$$
\aligned
\partial_t\theta&=-\nu A(1+\alpha A)^{-1}\theta\\&-(1+\alpha A)^{-1}B(\theta,u(t))
-(1+\alpha A)^{-1}B(u(t),\theta)=:L_{u(t)}\theta.
\endaligned
$$

Following the general theory~\cite{T}, the  coefficients
describing the contraction
of $n$-dimensional volumes transported by
the linearized equation  are defined  by the sums of the first $n$
global Lyapunov exponents, that is, by the numbers
$q(n)$:$$
q(n):=\limsup_{t\to\infty}\sup_{u(t)\in\mathscr A}
\frac1t\int_0^t \sup_{\{\theta_j\}_{j=1}^n}
\sum_{j=1}^n(L_{u(\tau)}\theta_j,\theta_j)_\alpha d\tau,
$$
where the inner supremum is taken over all systems $\{\theta_j(t)\}_{j=1}^n$,
which form a basis in the $n$-dimensional phase volume, transported by the
linearized equation along $u(t)$ and which
are orthonormal with respect to the scalar product
$(\cdot,\cdot)_\alpha$ in ${\mathbf H}^1$:
$$
(\theta_i,\theta_j)_\alpha=(\theta_i,\theta_j)+\alpha(\nabla\theta_i,\nabla\theta_j)=
\delta_{i\,j},
$$
and the middle  supremum is taken over all solutions $u(t)\in\mathscr A_\alpha$.

Then
\begin{equation}\label{long}
\aligned
\sum_{j=1}^n(L_{u(t)}\theta_j,\theta_j)_\alpha&=\sum_{j=1}^n(L_{u(t)}\theta_j,(1+\alpha
A)\theta_j)\\&
=
-\nu\sum_{j=1}^n\|\nabla\theta_j\|^2-
\sum_{j=1}^n((\theta_j,\nabla) u,\theta_j)\le\\&\le
-\nu\sum_{j=1}^n\|\nabla\theta_j\|^2+c_d\int_\Omega\rho(x)|\nabla u(t,x)|\,dx\\&
\le -\nu\sum_{j=1}^n\|\nabla\theta_j\|^2+c_d\|\nabla u(t)\|\|\rho\|,
\endaligned
\end{equation}
where
$$
\rho(x)=\sum_{j=1}^n|\theta_j(x)|^2
$$
and  (see \cite{IKZ22})
$$
c_d=\left\{
      \begin{array}{ll}
       \sqrt{\frac12}, & d=2, \\
       \sqrt{\frac23}, & d=3.
      \end{array}
    \right.
$$

For the first sum we have by orthonormality
$$
\aligned
-\sum_{j=1}^n\|\nabla\theta_j\|^2=-\frac1\alpha
\left(\alpha\sum_{j=1}^n\|\nabla\theta_j\|^2
+\sum_{j=1}^n\|\theta_j\|^2-\sum_{j=1}^n\|\theta_j\|^2\right)\\=
-\frac1\alpha n+\frac 1\alpha\sum_{j=1}^n\|\theta_j\|^2\le
-\frac1\alpha n+\frac 1{\lambda_1\alpha}\sum_{j=1}^n\|\nabla\theta_j\|^2,
\endaligned
$$
which gives that
\begin{equation}\label{n}
-\sum_{j=1}^n\|\nabla\theta_j\|^2\le -\frac1{\alpha+\lambda_1^{-1}}n.
\end{equation}

It remains to use the estimates for $\|\rho\|$ from \cite{Lieb90, IKZ22}, the only ones
that depend on the dimension $d$ and on $\alpha$ in the estimates of the numbers $q(n)$:
\begin{equation}\label{Liebrho}
\aligned
&\|\rho\|\le
              \frac1{2\sqrt{\pi}}\frac{n^{1/2}}{\alpha^{1/2}}, \quad d=2, \\
&\|\rho\|\le\frac1{2\sqrt{\pi}}\frac{n^{1/2}}{\alpha^{3/4}}, \quad d=3,
\endaligned
\end{equation}
and estimate \eqref{apr2}. (The fact that both numerical constants are $1/(2\sqrt{\pi})$  is just a coincidence.) This finally gives that
$$
\aligned
q(n)\le-\frac\nu{\alpha+\lambda_1^{-1}}\cdot n
+\frac{n^{1/2}}{\sqrt{8\pi\alpha\lambda_1}}\frac{\|g\|}\nu,\quad d=2,\\
q(n)\le-\frac\nu{\alpha+\lambda_1^{-1}}\cdot n
+\frac{n^{1/2}}{\sqrt{6\pi\alpha^{3/2}
\lambda_1}}\frac{\|g\|}\nu,\quad d=3.
\endaligned
$$

Now, according to the results of \cite{Ch-I} (see also~\cite{Ch-I2001}), the numbers $n^*$ for which the right-hand sides of the inequalities for the numbers $q(n)$ are negative will be upper bounds
of the fractal dimension for $d=2$ and $d=3$.
This proves~\eqref{dim23}
\end{proof}

\begin{remark}
{\rm
The key role in the proof is played by inequalities~\eqref{Liebrho}, that is similar to
the role of the Lieb--Thirring inequalities in the classical Navier--Stokes system. They were originally proved in \cite{LiebJFA} for $\mathbb R^2$. It is not difficult to obtain explicit estimates of the constants in this case. It is more difficult to do this for the torus $\mathbb T^2$ and the sphere $\mathbb S^2$ \cite{Lieb90,ZIMS}.
These inequalities are also valid for the $L_p$-norms of $\rho$
and as a consequence yield the Gagliardo--Nirenberg inequalities for the embedding $H^1\hookrightarrow L_p$
on the torus and on the sphere.
}
\end{remark}

\begin{remark}
{\rm
As usual, an important question arises about the optimality of the obtained upper bounds
as
$\alpha\to0^+$. Of course, we are talking about the three-dimensional case, since
the obtained upper bound  in the two-dimensional case is not good at all and will be discussed in the next section.

A very close $\alpha$-model, namely, the Bardina system with dissipation,
which is obtained from \eqref{DEalpha} by replacing the bounded operator
$A(1+\alpha A)^{-1}$ with the identity operator,
\begin{equation}\label{Bardina}
\partial u+\gamma u+(1+\alpha A)^{-1}B(u,u)=(1+\alpha A)^{-1}g,
\end{equation}
was considered in~\cite{IZLap70, IKZ22}. In these works, optimal two-sided estimates for the dimension of the attractors of the Bardina system with dissipation on the two-dimensional and three-dimensional torus were obtained. Moreover, the lower bound for $\mathbb T^2$ is essential, and the lower bound for $\mathbb T^3$ is obtained by Squire's theorem.
However, the fundamental difference between \eqref{NSV}  and \eqref{Bardina}
is that in the first case, for $\alpha=0$, we have a classical Navier--Stokes system
with a finite-dimensional attractor, while system \eqref{Bardina}
has an attractor whose dimension grows as $\alpha^{-2}$ for $\alpha\to0$
(if $g\in L_2$).
}
\end{remark}

Note also that in the three-dimensional case our upper bound has the same order of growth  $\alpha^{-3/2}$ as $\alpha\to0$ as in \cite{CotiGal},
however, the order of growth with respect to  $G$ is significantly smaller ($G^2$ instead of $G^6$).
This raises the question of whether the growth of order
$\alpha^{-3/2}$ is optimal.

In fact, it is not. Below, in the three-dimensional case, we prove an estimate of the form
\begin{equation}\label{3/4}
 \dim\mathscr A_\alpha\preceq \frac1{(\alpha\lambda_1)^{3/4}}\cdot G^4.
\end{equation}
Finding the optimal growth rate of the attractor dimension as $\alpha\to 0^+$ is apparently not simple.

In the proof of the estimate \eqref{3/4}
(as well as  in the  theory of Navier--Stokes equations)
the key role  is played by the Lieb--Thirring inequality \cite{lthbook, Lieb, LT}.

\begin{theorem}
Let $\Omega\subseteq\mathbb R^d$  or $\Omega=\mathbb T^d=[0,L]^d$,
and let the system
$\{u_j\}_{j=1}^n\in
	\mathbf H^1$, $\div u_j=0$ be suborthonormal in
	$L_2$: that is, for every $\xi\in \mathbb R^n$
 $$
 \sum_{i,j=1}^n\xi_i\xi_j(u_i,u_j)\le\sum_{i=1}^n\xi_i^2
$$
Then
$$
\rho(x):=\sum_{j=1}^N|u_j(x)|^2
$$
satisfies the inequality
\begin{equation}\label{CLT}
\int_{\Omega}\rho(x)^{1+2/d}dx
\le \mathrm{c}_{\mathrm {LT}}\sum_{j=1}^N\|\nabla u_j\|^2,
\end{equation}
where (see \cite{lthbook} and \cite{ILZJFA}, respectively)
$$\mathrm{c}_{\mathrm {LT}}\le
\left\{
  \begin{array}{ll}
   \frac1{2\pi}\cdot1.456\dots\,.,
   \ \Omega\subseteq\mathbb R^2 \\ \\
    \frac{3\pi}{32},\ \Omega=\mathbb T^2.
  \end{array}
\right.
$$
In the three-dimensional case (see \cite{lthbook} and \cite{IJST11}, respectively)
$$\mathrm{c}_{\mathrm {LT}}\le
\left\{
  \begin{array}{ll}
    \frac56\frac{2^{1/3}}{\pi^{4/3}}\cdot(1.456\dots)^{2/3}\,,
    \Omega\subseteq\mathbb R^3  \\ \\
    \frac53\left(\frac2\pi\right)^{2/3},\ \Omega=\mathbb T^3.
  \end{array}
\right.
$$
\end{theorem}

\begin{remark}
{\rm

In the scalar case, the constant $\mathrm{c}_{\mathrm {LT}}(d)$
is related to the constant $\mathrm L_{1,d}$ by the equality~\cite{lthbook}
$$
\mathrm{c}_{\mathrm {LT}}(d)=(2/d)(1+d/2)^{1+2/d}\mathrm L_{1,d}^{2/d},
$$
where $\mathrm L_{1,d}$ is the constant in the Lieb--Thirring bound
for the negative trace of the Schrodinger operator
$-\Delta-V(x)$, $V(x)\ge0$ in $\mathbb R^d$:
$$
\Tr(-\Delta-V)_-\le\mathrm L_{1,d}\int_{\mathbb R^d}V(x)^{1+d/2}dx.
$$
In turn, the best  to date estimate of the constant
$\mathrm L_{1,d}$ is (see \cite{lthbook})
$\mathrm L_{1,d}\le\mathrm L_{1,d}^\mathrm{cl}\cdot 1.456\dots$, where
$\mathrm L_{1,d}^\mathrm{cl}=\frac{\omega_d}{(2\pi)^d}\frac2{d+2}$.
Moreover, in the vector case, the constant $\mathrm L_{1,d}$
goes over to $d\mathrm L_{1,d}$. In the two-dimensional divergence-free
case, the constant does not double.
}
\end{remark}

\begin{remark}
{\rm
The orthonormal case was considered in the above-mentioned works.
The fact that in the transition from an orthonormal to a suborthonormal
system the constant in the Lieb--Thirring inequality does not
increase is shown in \cite{I2005}.
}
\end{remark}

\begin{theorem}\label{Th:2}
The following upper bound for the dimension holds in the 3D case
\begin{equation}\label{dim_alpha}
\dim\mathscr A_\alpha\le C(1+\alpha\lambda_1)G^{5/2}
 \left(\frac{(1+\alpha\lambda_1)}{(\alpha\lambda_1)^{3/4}}G^{3/2}+1\right),
 \quad G=\frac{\|g\|}{\lambda_1^{3/4}\nu^2}.
\end{equation}
\end{theorem}
\begin{proof}

First of all we point out that we cannot give all the constants in
the proof in explicit form, therefore, the same letter $C$ below
denotes all the dimensionless constants that occur.

From the dissipative estimate \eqref{apr1} we have a uniform
estimate in $t$ on the attractor
\begin{equation}\label{diss}
\|\nabla u(t)\|^2\le\frac{(\alpha\lambda_1+1)}{\alpha\nu^2\lambda_1^2}\|g\|^2,
\end{equation}
and also estimate~\eqref{u} that is independent of $\alpha$.

Using the Sobolev inequality $H^1\hookrightarrow L_6$ and
H\"older's inequality we obtain
$$
\|(u,\nabla)u\|_{L_{3/2}}\le\|u\|_{L_6}\|\nabla u\|\le C\|\nabla u\|\|\nabla u\|.
$$
Next, using the embedding $L_{3/2}\hookrightarrow H_A^{-1/2}$
(following from the dual embedding $H_A^{1/2}\hookrightarrow L_3$),
the boundedness of the Helmholtz--Leray projection in $L_p$-spaces,
as well as \eqref{u} and \eqref{diss}, we obtain
$$
\limsup_{t\to\infty}\frac1t\int_0^t\|B(u(s)),u(s))\|^2_{H_A^{-1/2}}ds
\le C\frac{\alpha\lambda_1+1}{\alpha\nu^4\lambda_1^3}
\|g\|^4,
$$
where $H^s_A:=D(A^{s/2})$ is the scale of Hilbert spaces
generated by the powers of the Stokes operator $A$. In addition,
using  estimate~\eqref{diss}
twice, we obtain a uniform estimate in $t$
\begin{equation}\label{Buni-t}
\|B(u(t)),u(t))\|^2_{H_A^{-1/2}}\le
 C\frac{(\alpha\lambda_1+1)^2}{\alpha^2\nu^4\lambda_1^4}
\|g\|^4,
\end{equation}

Let us consider our equation as  ``linear'':
$$
\partial_tu+\alpha A\partial_tu+\nu Au=H(t):=-B(u(t),u(t))+g
$$
and take the scalar product with $A^{1/2}u$:
$$
\partial_t\bigl(\|u\|_{H^{1/2}}^2+\alpha\|u\|_{H^{3/2}}^2\!\bigr)
+2\nu\|u\|_{H^{3/2}}^2\!\le2\|B\|_{H^{-1/2}}\|u\|_{H^{3/2}}
+2\|g\|\|\nabla u\|,
$$
or
$$
\partial_t\bigl(\|u\|_{H^{1/2}}^2+\alpha\|u\|_{H^{3/2}}^2\!\bigr)
+\nu\|u\|_{H^{3/2}}^2\!\le\nu^{-1}\|B\|_{H^{-1/2}}^2
+2\|g\|\|\nabla u\|.
$$
Acting as in Proposition~\ref{Prop:1} and using
\eqref{Buni-t}, we obtain uniform in  $t$ boundednes of
$\|u(t)\|^2_{H^{3/2}}$. After that integrating in $t$ the last equation
 and taking time average we obtain
\begin{multline*}
\limsup_{t\to\infty}\frac1t\int_0^t\|u(s)\|^2_{H^{3/2}}ds\le
\nu^{-2}\limsup_{t\to\infty}\frac1t\int_0^t\|B(u(s)),u(s))\|^2_{H_A^{-1/2}}ds\ \ \\
+2\nu^{-1}\|g\|\limsup_{t\to\infty}\frac1t\int_0^t\|\nabla u(s)\|ds
\le C\frac{\alpha\lambda_1+1}{\alpha\nu^6\lambda_1^3}
\|g\|^4+C\frac1{\nu^2\lambda_1^{1/2}}\|g\|^2.
\end{multline*}

We now use Gagliardo--Nirenberg inequality
 $$
 \|v\|_{L_{5/2}}\le C\|v\|^{2/5}\|v\|_{H^{1/2}}^{3/5}
 $$
for $v:=\nabla u$,
which together with last estimate and  \eqref{u}  gives that
\begin{multline*}
 \limsup_{t\to\infty}\frac1t\int_0^t\|\nabla u(s)\|^2_{L_{5/2}}\, ds\le
 \\\le C\left(\frac{\alpha\lambda_1+1}{\alpha\nu^6\lambda_1^3}\|g\|^4+
 \frac1{\nu^2\lambda_1^{1/2}}
 \|g\|^2\right)^{3/5}\left(\frac1{\nu^2\lambda_1}\|g\|^2\right)^{2/5}
\end{multline*}
and therefore
\begin{multline*}
 \limsup_{t\to\infty}\frac1t\int_0^t\|\nabla u(s)\|_{L_{5/2}}\, ds\le
 \\\le C\left(\frac{\alpha\lambda_1+1}{\alpha\nu^6\lambda_1^3}\|g\|^4+
 \frac1{\nu^2\lambda_1^{1/2}}\|g\|^2\right)^{3/10}
 \left(\frac1{\nu^2\lambda_1}\|g\|^2\right)^{2/10}=:\Xi.
\end{multline*}

Now, acting as in  \eqref{long}, we obtain
$$
\aligned
\sum_{j=1}^n(L_{u(t)}\theta_j,&\theta_j)_\alpha
\le
-\nu\sum_{j=1}^n\|\nabla\theta_j\|^2+c_3\int_\Omega\rho(x)|\nabla u(t,x)|\,dx
\\&\le
 -\nu\sum_{j=1}^n\|\nabla\theta_j\|^2+c_3\|\rho\|_{L_{5/3}}\|\nabla u(t)\|_{L_{5/2}},
 \endaligned
 $$
 which gives after averaging in  $t$, using the Lieb--Thirring
 inequality
 \eqref{CLT} and  \eqref{n} that
 $$
 \aligned
q(n)\le
 -\nu\sum_{j=1}^n\|\nabla\theta_j\|^2+
 c_3\left(\mathrm{c}_{\mathrm {LT}}\sum_{j=1}^n\|\nabla\theta_j\|^2\right)^{3/5}
\!\!\! \cdot\Xi
 \\\le
 -\frac\nu 2\sum_{j=1}^n\|\nabla\theta_j\|^2+
 C\nu^{-3/2}
\Xi^{5/2}\le-\frac\nu 2\frac{\lambda_1}{\alpha\lambda_1+1}n+
C\nu^{-3/2}
\Xi^{5/2}.
\endaligned
$$
Therefore
$$
\dim_F\mathscr A_\alpha\le C
\frac{\alpha\lambda_1+1}{\nu^{5/2}\lambda_1}
\left(\frac{\alpha\lambda_1+1}{\alpha\nu^6\lambda_1^3}\|g\|^4+
 \frac1{\nu^2\lambda_1^{1/2}}\|g\|^2\right)^{3/4}\!\!\!
 \left(\frac1{\nu^2\lambda_1}\|g\|^2\right)^{1/2}.
$$
In terms of  $\alpha\lambda_1$ and
$G$ this is precisely \eqref{dim_alpha}.
\end{proof}

\begin{corollary}
For small $\alpha\lambda_1$ and large $G$, the dimension
estimates in Theorems~\ref{Th:1},\ref{Th:2}
can be combined in the following symmetric
form
$$
\dim_F\mathscr A_\alpha\preceq(\alpha\lambda_1)^{-3/4}\cdot G^2
\min\bigl[(\alpha\lambda_1)^{-3/4}, G^2\bigr].
$$
\end{corollary}
 \setcounter{equation}{0}
\section{Two-dimensional case} \label{sec3}

Note that the dimension estimate in Theorem~\ref{Th:1} for $d=2$ blows up as
$\alpha\to0^+$, which should not happen. To obtain an estimate in the
two-dimensional case that is bounded as $\alpha\to0^+$, we again need
the Lieb--Thirring inequality.

\begin{theorem}
Both for $\Omega\subset \mathbb R^2$
and $\Omega=\mathbb T^2$ the following estimate holds:
\begin{equation}\label{G2}
\dim_F\mathscr A_\alpha\le\frac{(\alpha\lambda_1+1)\mathrm{c}_{\mathrm {LT}}}2G^2.
\end{equation}
\end{theorem}
\begin{proof}
As before~\eqref{long} holds. Then
\begin{equation}\label{long2}
\aligned
&\sum_{j=1}^n(L_{u(t)}\theta_j,\theta_j)_\alpha
\le -\nu\sum_{j=1}^n\|\nabla\theta_j\|^2+c_2\|\nabla u(t)\|\|\rho\|\\&
\le
-\nu \sum_{j=1}^n\|\nabla\theta_j\|^2
+c_2\|\nabla u(t)\|\left(\mathrm{c}_{\mathrm {LT}}
\sum_{j=1}^n\|\nabla\theta_i\|^2\right)^{1/2}
\\&
\le -\frac\nu  2\sum_{j=1}^n\|\nabla\theta_j\|^2+
\frac{\mathrm{c}_{\mathrm {LT}}\|\nabla u(t)\|^2}{4\nu}.
\endaligned
\end{equation}
Averaging in $t$, using \eqref{u} and \eqref{n}, we obtain
\begin{equation*}\label{q1}
q(n)\le-\frac\nu {2(\alpha+\lambda_1^{-1})}\cdot n
+\frac{\mathrm{c}_{\mathrm {LT}}\|g\|^2}{4\lambda_1\nu^3},
\end{equation*}
which gives \eqref{G2}.
\end{proof}
The  dimension estimate \eqref{G2} depends quadratically on the
number $G$, although it is well known \cite {BV, T} that for the classical
two-dimensional Navier--Stokes system (i.e., for $\alpha=0$)
$\dim_F\mathscr A\preceq~G$. Namely, this estimate
can conveniently be written  \cite{Lieb90} in terms
of the dimensionless number $\mathcal G$:
\begin{equation}\label{G1}
\dim_F\mathscr A\le\frac{\mathrm{c}_{\mathrm {LT}}^{1/2}}{2\sqrt{2}\pi}\mathcal G<
0.055\cdot \mathcal G,
\quad \mathcal G,
\quad \mathcal G=\frac{\|f\||\Omega|}{\nu^2}.
\end{equation}
The number $\mathcal G$ is obviously easier to calculate,
 $G\le(2\pi)^{-1}\mathcal G$, and finally,
for a torus these two numbers are simply proportional.
Thus, the obtained estimate \eqref{G2} is bounded as  $\alpha\to0$,
but is one order of magnitude worse than \eqref{G1}.
In fact, the estimate of the dimension is valid,
which for small $\alpha$ is linear in $G$ and/or
$\mathcal G$.

\begin{theorem}
Let $\Omega\subset \mathbb R^2$,  $|\Omega|<\infty$ or let
$\Omega=\mathbb T^2$.  Then in the first case for
$$
\alpha\in(0,\alpha_0],\qquad\alpha_0=\frac{|\Omega|}{2\pi}\cdot
\frac1{\mathcal G}
$$
it holds that
\begin{equation}\label{dim2}
\dim_F\mathscr A_\alpha\le\frac{\mathrm{c}_{\mathrm {LT}}^{1/2}}{\sqrt{2}\pi}\mathcal G<
0.109\cdot \mathcal G.
\end{equation}

In the case of a torus for
$$
\alpha\in(0,\alpha_0],\qquad\alpha_0=\frac{|\mathbb T^2|}{\pi^2}\cdot
\frac1{\mathcal G}
$$
it holds that
\begin{equation}\label{dimT2}
\dim_F\mathscr A_\alpha\le\frac1{\pi^2}
\left(\frac{\mathrm{c}_{\mathrm {LT}}}2\right)^{1/2}\mathcal G\le
\frac{3^{1/2}}{8\pi^{3/2}}\,\mathcal G
<0.039\cdot \mathcal G.
\end{equation}
\end{theorem}
\begin{proof}
We only need to estimate $\sum_{j=1}^n\|\nabla\theta_j\|^2$ in \eqref{long2}.
Let us first consider the case $\Omega\subset\mathbb R^2$.
We set
$$
\varphi_j:=(1+\alpha A)^{1/2}\theta_j,
$$
where $A$ is the Stokes operator. Then the system $\{\varphi_j\}_{j=1}^n$
is orthonormal in $L_2$:
$$
(\varphi_i,\varphi_j)=(\theta_i,\theta_j)+
\alpha(\nabla\theta_i,\nabla\theta_j)=\delta_{i\,j}.
$$
Therefore it follows from the variational principle that
$$
\aligned
&\sum_{j=1}^n\|\nabla\theta_j\|^2=
\sum_{j=1}^n\left(\varphi_j,A(1+\alpha A)^{-1}\varphi_j\right)\\
\ge&\sum_{j=1}^n\nu_j=\sum_{j=1}^n\frac{\lambda_j}{1+\alpha\lambda_j}
\ge\frac{2\pi}{|\Omega|}\sum_{j=1}^n\frac j{1+\alpha'j},
\endaligned
$$
where
$$
\alpha':=\frac{2\pi}{|\Omega|}\cdot\alpha,
\quad \alpha'\le \mathcal G^{-1},
$$
and where $\{\nu_j\}_{j=1}^\infty$ are the nondecreasing  eigenvalues of the operator
$A(1+\alpha A)^{-1}$, which are obviously expressed
in terms of the nondecreasing  eigenvalues $\lambda_j$ of the Stokes operator $A$.
We also took into account that the function $t\to t/(1+\alpha t)$ is
monotone increasing and
a (non sharp) lower bound for $\lambda_j$:
\begin{equation}\label{lbstokes}
\lambda_j\ge\frac{2\pi}{|\Omega|}j\,.
\end{equation}
following from the Li--Yau-type bound for the Stokes operator \cite{FA09}
$$
\sum_{j=1}^m\lambda_j\ge\frac{2\pi}{|\Omega|}m^2,
$$
since $m\lambda_m\ge\sum_{j=1}^m\lambda_j$.

Let $n\le \mathcal G$.
Then for all $j=1,\dots,n$ it holds that $1+\alpha'j\le2$,
and therefore
$$
\sum_{j=1}^n\|\nabla\theta_j\|^2\ge\frac{2\pi}{|\Omega|}\cdot\frac{n^2}4,
$$
and instead of \eqref{long2}, taking into account that $\lambda_1\ge2\pi/|\Omega|$,
we obtain
\begin{equation}\label{qOmega}
q(n)\le -\frac{\pi}4\frac\nu{|\Omega|}  n^2
+\frac{\mathrm{c}_{\mathrm {LT}}}{8\pi}\frac{\|g\|^2|\Omega|}{\nu^3}.
\end{equation}
For
$$
n^*=\frac{\mathrm{c}_{\mathrm {LT}}^{1/2}}{\sqrt{2}\pi}\mathcal G
$$
we have $q(n^*)\le0$. Moreover,
$$
\frac{\mathrm{c}_{\mathrm {LT}}^{1/2}}{\sqrt{2}\pi}<1
\quad\Rightarrow\quad n^*<\mathcal G,
$$
and we are on time in getting \eqref{dim2}.
The proof is complete for $\Omega\subset \mathbb R^2$.

In the case of a torus (without loss of generality, we assume that
$\mathbb T^2=[0,2\pi]^2$) it is only necessary to check
the estimate~\eqref{lbstokes}, where
$\lambda_j$ are the eigenvalues of the Laplace operator,
for which even slightly more than~\eqref{lbstokes} is true
(see Proposition~\ref{Prop:lambda}):
\begin{equation}\label{1.15}
\lambda_j\ge \frac j4=j\frac{\pi^2}{|\mathbb T^2|}.
\end{equation}

Proceeding as in the case of the domain $\Omega$, we obtain that for
any $n\le \mathcal G$ and $\alpha\le(|\Omega|/\pi^2)\cdot \mathcal G^{-1}$
we have
$$
\sum_{j=1}^n\|\nabla\theta_j\|^2\ge\frac{\pi^2}{|\mathbb T^2|}\cdot\frac{n^2}4.
$$
Using \eqref{u}, where $\lambda_1=4\pi^2/|\mathbb T^2|$,
instead of~\eqref{qOmega} we get
$$
q(n)\le -\frac{\pi^2}8\frac\nu{|\mathbb T^2|} n^2
+\frac{\mathrm{c}_{\mathrm {LT}}}{4}\frac{\|g\|^2|\mathbb T^2|}{4\pi^2\nu^3},
$$
 which  again on time implies
\eqref{dimT2},
since $(\mathrm{c}_{\mathrm {LT}}/2)^{1/2}/{\pi^2}<1$.
\end{proof}

\begin{remark}
{\rm
Estimate \eqref{G2} does not grow as $\alpha\to0^+$, while
one order of magnitude slower growing estimate \eqref{G1} is uniform
in $\alpha$ on finite intervals of length of order~$G^{-1}$.
 }
\end{remark}

\begin{corollary}\label{Th:T2clnolog}
For the classical Navier--Stokes system on the
2D torus ($\alpha=0$)
it holds
$$
\dim_F\mathscr A \le\frac{\sqrt{\mathrm c_{\mathrm{LT}}}}{2\pi^2}
\,\mathcal G\le \frac{3^{1/2}}{2^{7/2}\pi^{3/2}}\,\mathcal G\le0.028\cdot \mathcal G.
$$
\end{corollary}
\begin{proof}
The system $\{\theta_j\}_{j=1}^n$ is orthonormal in
 $L_2$, therefore
$$
\sum_{j=1}^n\|\nabla\theta_j\|^2\ge
\sum_{j=1}^n\lambda_j\ge\frac{\pi^2}{|\mathbb T^2|}\cdot\frac{n^2}2.
$$
\end{proof}

It remains to be seen what happens
with the well-known logarithmically sharp estimate
$$
\dim_F\mathscr A\le cG^{2/3}(\ln G+1)^{1/3}
$$
for the Navier--Stokes system on the two-dimensional torus
\cite{CFT,T} (see also \cite{DoerGib}, \cite{INon})
when going over  to the Navier--Stokes--Voight system. As we shall see,
the situation is almost the same as in the case of a bounded domain,
namely, the following theorem holds.

\begin{theorem}\label{Th:logT2}
Let
$\Omega=\mathbb T^2$.
Then for
$$
\alpha\in[0,\alpha_0],\qquad\alpha_0=\frac{|\Omega|}{\pi^2}\cdot
\frac1{\mathcal G}
$$
it holds that
\begin{equation}\label{dimT2log}
\dim_F\mathscr A_\alpha\le
\min\left[0.039\cdot \mathcal G,\ \ 7.46\cdot
\mathcal G^{2/3}(\ln \mathcal G+5.74)^{1/3}\right].
\end{equation}
\end{theorem}
\begin{proof}
We apply the operator $\rot$ to the first equation in~\eqref{DEalpha}
and obtain the scalar equation
\begin{equation}\label{vort}
(1-\alpha \Delta)\partial_t \omega+ u\cdot\nabla \omega -\nu\Delta \omega
=\rot g,
\end{equation}
where
$$
\omega=\rot u,\qquad u=\nabla^{\perp}\Delta^{-1}\omega.
$$

Applying  $(1-\alpha\Delta)^{-1}$ to \eqref{vort}
we obtain the analogue of~\eqref{NSV}
\begin{equation}\label{vort1}
\partial_t \omega+ (1-\alpha \Delta)^{-1}(u\cdot\nabla \omega)
 -\nu\Delta(1-\alpha \Delta)^{-1} \omega
=(1-\alpha \Delta)^{-1}\rot g,
\end{equation}
and taking the scalar product with $(1-\alpha\Delta)\omega$ we obtain
$$
\partial_t(\|\omega\|^2+\alpha\|\nabla\omega\|^2)
+2\nu\|\nabla\omega\|^2=2(\rot g,\omega)=2(g,\rot\omega)\le
2\|g\|\|\nabla\omega\|,
$$
which, as before, gives that
$$
\limsup_{t\to\infty}\frac 1t\int_0^t\|\nabla\omega(\tau)\|^2d\tau\le
\frac{\|g\|^2}{\nu^2}\,.
$$
The linearized equation corresponding to  \eqref{vort1} is
$$
\partial_t\varphi=-(1-\alpha\Delta)^{-1}\left[(u(t)\cdot\nabla\varphi)
+(\nabla^\perp\Delta^{-1}\varphi\cdot\nabla\omega)-\nu\Delta\varphi\right]
=:L_{\omega(t)}\varphi.
$$

Then we consider the  $n$-trace of the operator $L_{\omega(t)}$
with respect to the
system $\{\varphi_j\}_{j=1}^n$ that is  orthonormal in
 $\dot H^1_\alpha$
$$
(\varphi_i,\varphi_j)_\alpha=(\varphi_i,\varphi_j)+
\alpha(\nabla\varphi_i,\nabla\varphi_j)=\delta_{i\,j}.
$$
Setting $v_j:=\nabla^\perp\Delta^{-1}\varphi_j$ we find that
\begin{equation*}\label{long3}
\aligned
&\Tr_nL_{\omega(t)}:=\sum_{j=1}^n(L_{\omega(t)}\varphi_j,\varphi_j)_\alpha
= -\nu \sum_{j=1}^n\|\nabla\varphi_j\|^2
-\sum_{j=1}^n
\left(v_j\cdot\nabla\omega(t),\varphi_j\right)\\&
=
-\nu \sum_{j=1}^n\|\nabla\varphi_j\|^2
+\int_{\mathbb T^2}\biggl(\sum_{j=1}^n
|v_j|^2\biggr)^{1/2}\biggl(\sum_{j=1}^n
\varphi_j^2\biggr)^{1/2}|\nabla\omega(t)|dx\\&\le
-\nu \sum_{j=1}^n\|\nabla\varphi_j\|^2
+\|\rho\|_{L_\infty}^{1/2}\biggl(\sum_{j=1}^n
\|\varphi_j\|^2\biggr)^{1/2}\|\nabla\omega(t)\|
\\&
\le -\nu \sum_{j=1}^n\|\nabla\varphi_j\|^2
+\|\rho\|_{L_\infty}^{1/2}n^{1/2}\|\nabla\omega(t)\|,
\endaligned
\end{equation*}
where
$$
\rho(x)=\sum_{j=1}^n|v_j(x)|^2=\sum_{j=1}^n|\nabla\Delta^{-1}\varphi_j(x)|^2.
$$

We set
$$
|\mathbb T^2|\sum_{j=1}^n\|\nabla\varphi_j\|^2=:T(t)\quad
\text{and also}\quad\mathcal N(t)^2:=\frac1t\int_0^tT(\tau)d\tau.
$$
We again assume that
$n\le \mathcal G$.
It follows from~\eqref{1.15} that
$$
 T,\, \mathcal N^2\ge\frac{\pi^2}4n^2.
$$
We use the estimate for $\|\rho\|_{L_\infty}$ in Proposition~\ref{Prop:rho}
and set $\Lambda:=[T]+1$ in \eqref{rhoinf}:
$$
\|\rho\|_{L_\infty}^{1/2}\le 4\sqrt{2}\pi(\ln8eT)^{1/2}+4
<8\sqrt{2}\pi(\ln8eT+1)^{1/2}.
$$
Collecting the above we obtain
$$
\sum_{j=1}^n(L_{\omega(t)}\varphi_j,\varphi_j)_\alpha\le
\frac\nu{|\mathbb T^2|}
\left(-T+\mathrm k_1T^{1/4}(\ln T+\mathrm k_2)^{1/2}\|\omega(t)\||\mathbb T^2|\nu^{-1}\right),
$$
where
$$
\mathrm k_1=\sqrt{2/\pi}\,8\sqrt{2}\pi=16\sqrt{\pi},\quad \mathrm k_2=3\ln2+2.
$$
Next, averaging in $t$, using Cauchy--Schwartz inequality and
Jensen's inequality (for $T\to\sqrt{T}(\ln T+\mathrm k_2)$),
we obtain
\begin{multline*}
\frac1t\int_0^t\Tr_nL_{\omega(\tau)}d\tau\\\le
\frac{\nu\mathcal N^{1/2}}{|\mathbb T^2|}
\left(-\mathcal N^{3/2}+\sqrt{2}\mathrm k_1(\ln \mathcal N+\mathrm k_2/2)^{1/2}
\mathcal G(1+o_{t\to\infty}(1))\right).
\end{multline*}
Thus, $q(n)<0$, if for large $t$
$$
-\mathcal N^{3/2}+\mathrm K(\ln \mathcal N+\mathrm k_2/2)^{1/2}<0
,\quad \mathrm K=\sqrt{2}\mathrm k_1\,\mathcal G.
$$
This inequality holds if
$$
\mathcal N\ge \mathrm K^{2/3}(\ln\mathrm K+\mathrm k_2/2)^{1/3}.
$$
The last claim is equivalent to the implication
$$
\mathcal N^{3/2}<\mathrm K(\ln \mathcal N+\mathrm k_2/2)^{1/2}
\quad\Rightarrow\quad
\mathcal N\le\mathrm K^{2/3}(\ln \mathrm K+\mathrm k_2/2)^{1/3}.
$$
In fact, taking logarithms of the first inequality with the account that
$\frac12\ln(\ln \mathcal N+\mathrm k_2/2)<\ln\mathcal N$, we find that
$\ln\mathcal N<\ln\mathrm K$, which implies the second inequality.
Recalling that $n<\frac2\pi\mathcal N$, we finally obtain
$$
\aligned
\dim_F\mathscr A_\alpha\le
\frac2\pi(\sqrt{2}\mathrm k_1)^{2/3}\mathcal G^{2/3}
(\ln \mathcal G+\mathrm k_2/2+\ln(\sqrt{2}\mathrm k_1))^{1/3}\\=
\frac{16}{\pi^{2/3}}\mathcal G^{2/3}\biggl(\ln \mathcal G+\frac12\ln \pi+6\ln 2+1\biggr)^{1/3}\\<
7.46\cdot \mathcal G^{2/3}(\ln \mathcal G+5.74)^{1/3}.
\endaligned
$$
We now recall the condition $n\le\mathcal G$. It holds for large
$G$, namely, if
$$
\mathcal G\ge \mathcal G_0=6000,
$$
where $x\thickapprox6000$ is the root of the equation
 $x=7.46x^{2/3}(\ln x+5.24)^{1/3}$.
It remains to take into account estimate~\eqref{dimT2},
for which condition  $n\le\mathcal G$ is satisfied.
\end{proof}

\begin{corollary}\label{Th:logT2Th:logT2cl}
For the classical Navier--Stokes system
on the 2D torus ($\alpha=0$)
the following estimate for the dimension holds
\begin{equation}\label{T20min}
\aligned
&\dim_F\mathscr A\\&\le\min\left[\frac{3^{1/2}}{2^{7/2}\pi^{3/2}}\mathcal G,\
\frac{2^{10/3}}{\pi^{2/3}}\mathcal G^{2/3}
\biggl(\ln \mathcal G+\frac12\ln \pi+\frac{23}4\ln 2+1\biggr)^{1/3}\right]
\\&<
\min\left[0.028\cdot \mathcal G,\ 4.7 \mathcal G^{2/3}
\biggl(\ln \mathcal G+5.56\biggr)^{1/3}\right].
\endaligned
\end{equation}
\end{corollary}
\begin{proof}
We only have to take into account that
\eqref{1.15} implies that
$$
\mathcal N\ge\frac{\pi^2}2n^2
$$
and therefore the constant
$\mathrm k_1$ goes over to
$$
\mathrm k_1'=\frac{\sqrt[4]{2}}{\sqrt{\pi}} 8\sqrt{2}\pi=2^{15/4}\sqrt{\pi}.
$$
\end{proof}

\begin{remark}
{\rm
We observe that the minimum in~\eqref{dimT2log} is attained at
the second argument for $\mathcal G>2.6\cdot 10^8$, and in
\eqref{T20min} for $\mathcal G>1.14\cdot 10^8$.
}
\end{remark}

 \setcounter{equation}{0}
\section{Appendix. Estimate of function $\rho$ } \label{sec4}

For the estimate of  $\rho$ we need the following
fairly rough bounds for the eigenvalues of the
Laplacian on the torus.
\begin{proposition}\label{Prop:lambda}
The eigenvalues $\lambda_j$ of the Laplace(Stokes) operator
on the torus $\mathbb T^2=[0,2\pi]^2$ satisfy
\begin{equation}\label{bounds}
\lambda_j\ge\frac j4,\ j\ge1,\quad \lambda_j\le\frac j2,\ j\ge2.
\end{equation}
\end{proposition}
\begin{proof}
It follows from the Weyl asymptotics (in our case from the
asymptotics of the number of the points of the integer lattice  inside a circle) that
$
\lambda_j/j\to1/\pi
$,
and therefore inequalities \eqref{bounds}
can be violated only for finitely many $j$'s.

Let $N(E)$ be the number of the eigenvalues $\lambda_j$
less than or equal  $E$. Then $N(E)+1$ is the
number of the lattice points (including $(0,0)$) inside the circle
of radius $\sqrt{E}$.
The following two-sided estimate holds (see, for example, \cite{Kra}),
that follows from simple geometric considereations
$$
\pi(\sqrt{E}-\sqrt{2}/2)^2\le N(E)+1\le\pi(\sqrt{E}+\sqrt{2}/2)^2.
$$
Therefore
$$
N(E)\le 4E,\ E\ge29,\qquad N(E)>2E, \ E\ge15.
$$
By direct calculation we verify that $N(E)\le 4E$ for $E\ge~1$,
which gives the first inequality in~\eqref{bounds}. In the same way we check that
$N(E)<2E$, $E\ge2$, which gives the second inequality in~\eqref{bounds},
since the first eigenvalue $\lambda_1=1$ is excluded by hypothesis.
\end{proof}

\begin{proposition}\label{Prop:rho}
The following inequality holds for every $\Lambda\in\mathbb N$:
\begin{equation}\label{rhoinf}
\|\rho\|_{L_\infty}^{1/2}\le 4\sqrt{2}\pi(\ln4e\Lambda)^{1/2}+4
\Lambda^{-1/2}\left(|\mathbb T^2|\sum_{j=1}^n\|\nabla\varphi_j\|^2\right)^{1/2}.
\end{equation}
\end{proposition}
\begin{proof}
The system $\{\varphi_j\}_{j=1}^n$ is orthonormal in $\dot H^1_\alpha$,
and therefore it is suborthonormal in $L_2$, that is,
for any $\xi\in \mathbb R^n$
$$
\sum_{i,j=1}^n\xi_i\xi_j(\varphi_i,\varphi_j)\le\sum_{j=1}^n\xi_j^2.
$$
Inequality \eqref{rhoinf} (without explicit constants)
is known in the orthonormal case, so we only point out some necessary
refinements.

Thus, for $\theta\in \dot H^3$, using the Fourier series
$\theta(x)=\sum_{k\in\mathbb Z^2_0}e^{ik\cdot x}\theta_k$,
and taking into account that
 $\{|k|^2\}_{k\in\mathbb Z_0^2}=\{\lambda_j\}_{j=1}^\infty$,
we obtain
$$
\aligned
|\nabla\theta(x)|\le\sum_{|k|^2\le\Lambda}|k||\theta_k|+
\sum_{|k|^2>\Lambda}|k||\theta_k|\\=
\sum_{\lambda_j\le\Lambda}\lambda_j^{-1/2}\lambda_j|\theta_j|+
\sum_{\lambda_j>\Lambda}\lambda_j^{-1}\lambda_j^{3/2}|\theta_j|\\\le
\left(\sum_{\lambda_j\le\Lambda}\lambda_j^{-1}\right)^{1/2}
2\pi\|\Delta\theta\|+
\left(\sum_{\lambda_j>\Lambda}\lambda_j^{-2}\right)^{1/2}
2\pi\|\nabla\Delta\theta\|,
\endaligned
$$
where $1<\Lambda\in\mathbb N$ is an arbitrary integer parameter.

We estimate the sums using~\eqref{bounds}
$$
\aligned
&\sum_{\lambda_j\le\Lambda}\lambda_j^{-1}
<4\sum_{j\le4\Lambda}j^{-1}<4\ln4e\Lambda,\\
&\sum_{\lambda_j>\Lambda}\lambda_j^{-2}
<16\sum_{j\ge2\Lambda+1}j^{-2}<\frac8\Lambda.
\endaligned
$$
Let us substitute this into the inequality
$\theta=\sum_{j=1}\xi_j\Delta^{-1}\varphi_j$, $|\xi|=1$.
Since $|\xi|=1$, it follows that
$$
\|\sum_{j=1}^n\xi_j\varphi_j\|\le1\ \text{and}\
\|\sum_{j=1}^n\xi_j\nabla\varphi_j\|\le
\left(\sum_{j=1}^n\|\nabla\varphi_j\|^2\right)^{1/2},
$$
which gives
$$
\aligned
\bigl|\sum_{j=1}^n\xi_jv_j(x)\bigr|\le
4\pi(\ln4e\Lambda)^{1/2}+2\sqrt{2}
\Lambda^{-1/2}\left(|\mathbb T^2|\sum_{j=1}^n\|\nabla\varphi_j\|^2\right)^{1/2}=:R.
\endaligned
$$

We obtain that for $v_j(x)=v_j^1(x)\cdot e_1+v_j^2(x)\cdot e_2$
and every $\xi\in\mathbb R^n$  with $|\xi|=1$
it holds that
$$
\left(\sum_{j=1}^n\xi_jv_j^1(x)\right)^2+
\left(\sum_{j=1}^n\xi_jv_j^2(x)\right)^2\le R^2.
$$
It remains to set here at first $\xi_j=v_j^1(x)/(\sum_{j=1}^n (v_j^1(x))^2)^{1/2}$,
and then $\xi_j=v_j^2(x)/(\sum_{j=1}^n (v_j^2(x))^2)^{1/2}$,
which implies \eqref{rhoinf}:
$$
\rho(x)=\sum_{j=1}^nv_j^1(x)^2+\sum_{j=1}^nv_j^2(x)^2\le
2R^2.
$$\end{proof}

\subsection*{Acknowledgement}
This work was done with the financial support from the Russian
Science Foundation (grant  no. 23-71-30008)  (SZ),
and  within the framework of the state contract
of the Keldysh Institute of Applied Mathematics (AI).


\begin{thebibliography}{99}


\bibitem{BV} A.V. Babin, M.I. Vishik,
    \emph{Attractors of evolution equations}. Nauka,
    Moscow, \textrm 1988; \textrm {English transl.}
 North-Holland, Amsterdam, 1992.

\bibitem{CaoLunTiti}
Y. Cao, E. Lunasin,  E.S. Titi.
Global well-posedness of the three-dimensional viscous and
inviscid simplified Bardina turbulence models.
\textit{Commun. Math. Sci.} \textbf{4}:4 (2006), 823--848.



\bibitem{ChepMS}
V.V. Chepyzhov.
On approximation of the trajectory attractor for a
3D Navier-Stokes system by various hydrodynamic $\alpha$-models
\textit{Mat. Sb.} \textbf{207}:4 (2016), 143--172; translation in
\textit{Sb. Math.} \textbf{207}:3--4 (2016).

\bibitem{Ch-I2001}
 V.V. Chepyzhov,  A.A. Ilyin. {A note on the fractal dimension
of attractors of dissipative dynamical systems}. \textit{Nonlinear
Anal.} \textbf{44} (2001), 811--819.



\bibitem{Ch-I} V.V. Chepyzhov,  A.A. Ilyin. {On the fractal
    dimension of invariant sets; applications to Navier--Stokes
    equations}. \textit{Discrete Contin. Dyn. Syst.}
    \textbf{10}:1-2  (2004), 117--135.


\bibitem{ChepTitiVishik}
V.V. Chepyzhov, E.S. Titi, M.I. Vishik. On the convergence of
solutions of the Leray-$\alpha$ model to the trajectory attractor
of the 3D Navier--Stokes system.
 \textit{Discrete Contin. Dyn. Syst.}
 \textbf{17} (2007),  481--500.



\bibitem{CHOT}
A. Cheskidov, D.D. Holm, E. Olson and E.S. Titi,
On a Leray-$\alpha$ model of turbulence.
\textit{Royal Soc. A --- Math. Phys.  Eng. Sci.} {\bf 461} (2005), 629--649.




\bibitem{CFT} P. Constantin, C. Foias and R. Temam. On the
    dimension of the attractors in two--dimensional turbulence.
    \textit{Phisica D}  \textbf{30} (1988),  284--296.

\bibitem{CotiGal}
M. Coti Zelati, C.Gal.
Singular limits of Voigt models in fluid dynamics.
\textit{J. Math. Fluid Mech.} \textbf{17}:2 (2015),  233--259.



\bibitem{DoerGib} C. Doering and J. Gibbon.
 Note on the
Constantin--Foias--Temam attractor dimension estimate for
two--dimensional turbulence. \textit{Phisica D} \textbf{l48}
(1991),  471--480.


\bibitem{FHT} \textrm{C. Foias, D.D. Holm, E.S. Titi.}
\textrm{The three dimensional viscous Camassa--Holm equations, and their relation to the
Navier--Stokes equations and turbulence theory } \textit{J. Dynam. Differential Equations}
 \textbf{14}, 1--35 (2002).



\bibitem{lthbook} R.L. Frank, A. Laptev, T. Weidl.
    \textrm{Schr\"odinger Operators: Eigenvalues and Lieb--Thirring
    Inequalities}. --- (Cambridge Studies in Advanced Mathematics 200).
    Cambridge: Cambridge University Press, 2022.


\bibitem{IJST11}
A.A.Ilyin. Lieb--Thirring inequalities on some
manifolds.
{\it Journal of Spectral Theory.} {\bf 2}:1 (2012), 57--78.


\bibitem{INon} A.A. Ilyin. Navier--Stokes equations on the rotating
    sphere. A simple proof of the attractor dimension estimate,
    {\it Nonlinearity} 7 (1994), 31--39.

 \bibitem{Lieb90}
A.A. Ilyin, A.G. Kostianko, S.V.  Zelik. Applications of the
Lieb--Thirring and other bounds for orthonormal systems in
mathematical hydrodynamics. \textit{The Physics and Mathematics of
Elliott Lieb. The 90th Anniversary. Volume~I.} Edited by R. L.
Frank, A. Laptev, M. Lewin, R. Seiringer
 European Mathematical Society,  EMS  Publishing House
 Berlin 2022,  582--608.   

\bibitem{IKZ22}
 \textrm{A.A. Ilyin, A.G. Kostianko, S.V. Zelik.}
    \textrm{Sharp upper and lower bounds of the attractor
    dimension for   3D damped Euler--Bardina equations.}
    \textrm{Physica D} \textbf{432}  (2022),  133156.  

\bibitem{FA09}   A.A.Ilyin, On the spectrum of the Stokes operator,
{\it Funktsonal. Anal. i Prilozhen.}
{\bf 43}:4 (2009), 14--25;
English transl. in
{\it Funct. Anal. Appl.} {\bf 43}:4 (2009).

\bibitem{I2005}
 A.A.Ilyin. Lieb--Thirring
integral inequalities and their applications to the attractors of the
Navier-Stokes equations, {\it Mat. Sbornik} 196:1 (2005), 3-66;
 English transl. in {\it  Sb.
Math.} 196:1 (2005).


\bibitem{ILZJFA} A.A. Ilyin, A.A. Laptev, S.V. Zelik. Lieb--Thirring
    constant  on the sphere and on the torus. {\it J. Func.
    Anal.} {\bf 279} (2020), 108784.







\bibitem{IZLap70} A.A. Ilyin,  Zelik S.V.
 Sharp dimension  estimates of the attractor  of
 the damped  2D Euler-Bardina equations.
In book: \textrm{Partial Differential Equations, Spectral Theory,
    and Mathematical Physics}, European Math. Soc. Press, Berlin, ---  2021, P.209--229.


\bibitem{TitiVarga} V.K. Kalantarov, E.S. Titi.
Global attractors and determining modes for the 3D
 Navier--Stokes--Voight equations.
 \textit{Chin. Ann. Math. Ser. B} \textbf{30} (2009), 697--714.


\bibitem{Kra} E. Kr\"atzel. \textit{Lattice points.} Dordrecht,
    Kluwer, 1988


\bibitem{Lieb}
E. Lieb On characteristic exponents in turbulence. \textit{Comm.
Math.  Phys.} \textbf{92} (1984)
 473--480.   

 \bibitem{LiebJFA} E.H. Lieb.  An $L^p$ bound for the
    Riesz and Bessel potentials of orthonormal functions.
\textit{J. Func. Anal.}  \textbf{51}  (1983), 159--165.   

\bibitem{LT} E. Lieb, W. Thirring.
 Inequalities for the moments of the
eigenvalues of the Schr\"o\-dinger Hamiltonian and their relation
to Sobolev inequalities, \textrm{Studies in Mathematical Physics,
Essays in honor  of Valentine Bargmann}
 Princeton University Press,
 Princeton NJ, 269--303, 1976.  


\bibitem{Oskol}
A.P. Oskolkov.
The uniqueness and solvability in the large of boundary value problems for the
equations of motion of aqueous solutions of polymers.
\textit{Zap. Nau\v{c}n. Sem. Leningrad. Otdel. Mat.
Inst. Steklov. (LOMI)}, \textbf{38} (1973),  98--136;
translation in \textit{Boundary value problems of mathematical physics and
related questions in the theory of functions}, \textbf{7}.




\bibitem{T}\textrm{R. Temam.} \textit{Infinite Dimensional
    Dynamical Systems in Mechanics and Physics, \rm 2nd ed.}
    \textrm{New York},
\textrm{Sprin\-ger-Ver\-lag}, 1997.  



\bibitem{ZUMN} S.V. Zelik. Attractors. Then and now.
\textit{Russian Math. Surv.} \textbf{78}:4
(2003) 53--198.  

\bibitem{ZIMS} S.V. Zelik, A.A. Ilyin.  \textrm{On a class of
    interpolation inequalities on the 2D sphere.}
    \textit{Mat. Sb.} \textbf{214}:3 (2023),   120--134; English tansl.
    \textit{Sbornik: Mathematics} \textbf{214}:3  (2023).


\end{thebibliography}
\end{document}